\documentclass[12pt,oneside]{article}

\usepackage[T1]{fontenc}
\usepackage[utf8]{inputenc}
\usepackage{color}
\usepackage{braket}
\usepackage{hyperref}
\usepackage{comment}

\usepackage{amssymb,amsmath,amsthm}

\usepackage{microtype}

\usepackage[backend=biber,style=alphabetic,url=false]{biblatex}

\newcommand{\R}{\mathbb{R}}
\newcommand{\C}{\mathbb{C}}
\newcommand{\N}{\mathbb{N}}

\newcommand{\Dcal}{\mathcal{D}}
\newcommand{\Fcal}{\mathcal{F}}
\newcommand{\Ccal}{\mathit{C}}

\newcommand{\Cinf}{C^\infty}
\newcommand{\Scal}{\mathcal{S}}
\newcommand{\sob}{\mathsf{H}}
\newcommand{\del}{\partial}
\newcommand{\im}{\mathrm{i}}
\newcommand{\Bcal}{\mathcal{B}}
\newcommand{\Ecal}{\mathcal{E}}
\newcommand{\Bfrak}{\mathfrak{B}}
\renewcommand{\bar}{\overline}
\renewcommand{\Im}{\mathrm{Im}\,}

\DeclareMathOperator{\supp}{supp}

\DeclareMathOperator{\ran}{ran}
\DeclareMathOperator{\Tr}{Tr}
\DeclareMathOperator{\sgn}{sgn}

\newtheorem{lemma}{Lemma}
\newtheorem{theorem}{Theorem}
\newtheorem{corollary}{Corollary}

\author{Alessandro Pietro Contini}
\title{The Weyl law for the Dirichlet Laplacian}

\begin{document}
	\maketitle
	\tableofcontents
	
	\section*{Abstract}
	The purpose of this paper is to review the asymptotic distribution of eigenvalues of the Dirichlet Laplacian. We introduce and recall all the relevant spectral quantities and provide a proof based on the Fourier Tauberian Theorem.

	\section{Introduction}
	Determining the asymptotics distribution of the spectrum of operators is one of the longest-living subjects of intense research in analysis and mathematical physics. In its over 100 years long history (see the last Section for a partial historical account), many results have been proven or improved upon thanks to the introduction of more and more refined techniques, including but not limited to Dirichlet-Neumann bracketing, heat kernel analysis and Fourier Integral Operators. In this paper we aim to give an introduction to the modern proof of the Weyl law with remainder estimates, using the wave group approach on a compact Riemannian manifold with boundary. The rough sketch of the idea can be subsumed as follows:
	\begin{enumerate}
		\item One has a positive, scalar, second-order operator $P$ on the compact Riemannian manifold with boundary $\bar{M}$, acting first on $\Cinf_c(M)$. The extension of $P$ under Dirichlet boundary conditions, $\mathcal{P}$, has discrete spectrum, so the counting problem is well-defined.
		\item The counting function $N(\lambda)$ can be computed as the trace of the spectral family $E_\lambda$ of $\mathcal{P}$, that is, as the integral over $M$ of its distributional kernel $e(x,y,\lambda)$. Thus, estimating $e$ leads to explicit bounds on $N$.
		\item $e$ can be related via Fourier transform to $\cos(t\sqrt{\mathcal{P}})$, which is an inverse to the wave operator with Cauchy data $u|_{t=0}\in L^2$, $\dot u|_{t=0}=0$.
		\item The cosine group above can be approximated by the Hadamard parametrix construction, and the error is a smooth function. Furthermore, the first term in the approximation is (essentially) the Fourier transform of the surface measure of a sphere of radius $|t|$.
		\item Thus we have, near the diagonal of $M\times M$, a bound on the Fourier transform of the difference between the spectral function and the surface measure above. If we stay at distance $1/\sqrt{\lambda}$ from the boundary (which can be achieved using a localising function in Fourier space), the Fourier Tauberian Theorem allows to deduce a bound on $e$ in terms of the Hadamard coefficients.
		\item Near the boundary, a direct estimate of $e$ provides the required asymptotic behaviour.		
	\end{enumerate}
	For the sake of clarity and to avoid discouraging a non-specialist reader, we focus on the Laplace operator $P=\Delta$ and dedicate extensive space to the problem of constructing the Dirichlet extension on the Sobolev spaces of extendible distributions. 
	
	Thus, Section 2 is an account of the theory of Sobolev spaces on manifolds with boundary. We caution the reader that we use the \textit{complex interpolation method} as a ``black box'', since we believe that discussing it wouldn't contribute to the understanding of the topic. We refer the interested reader to \cite{taylor2023PDE1}, to which the whole Section is inspired, for a complete discussion. The most important result for what follows is the Poincaré inequality (Theorem \ref{theo:poincare inequality}), which is the main reason why the spectrum of the Dirichlet extension is discrete. A reader already familiar with such objects can safely skip this, keeping in mind that every occurrence of Sobolev spaces in the rest of the paper is to be taken as consisting of extendible distributions. 
	
	Section 3 deals with the problem of extending the Laplacian from $\Cinf_c(M)$ to the correct subspace of $L^2(M)$ in order to encode the homogeneous Dirichlet boundary condition (inhomogeneous boundary conditions can be handled by the theory of boundary layer potentials, see \cite{mclean_strongly_2000}). In particular, we prove that, when acting on extendible Sobolev distributions, the Dirichlet Laplacian $\Delta_D$ has discrete spectrum, so that the counting problem is well-defined. We structure the discussion after \cite{taylor2023PDE1} and \cite{borthwick_2020_spectral_theory}.
	
	Section 4 introduces the spectral function, namely the kernel of the spectral family of the Dirichlet Laplacian. There, we prove some useful estimates and analyse its cosine transform following \cite{hormander1994analysispseudodifferential}. The main result here are the estimates of Theorem \ref{theo:direct estimate of spectral function}.
	
	In Section 5, we recall the properties of Riesz distributions and show how to use them to solve the wave equation for $\Delta_D$. Together with the analysis of the cosine transform of the preceding Section, this gives explicit controls on the Fourier transform of the spectral measure. The structure of the discussion in this Section, as well as in the following one, is taken from \cite{hormander1994analysispseudodifferential}. The interested reader is invited to consult \cite{baer_wave_lorentzian_2007} for a more thorough discussion with geometric applications.
	
	Section 6 wraps up the discussion with the use of a Tauberian theorem to derive bounds for the spectral function from the control on its Fourier transform derived in Sections 4 and 5. This gives the local Weyl law of Theorem \ref{theo:local weyl law} and its global version of Corollary \ref{cor:weyl asymptotics}. We refer the reader to \cite{safarov_fourier_tauberian_2001} for an explicit approach.
	
	Finally, in Section 7 we give an historical account of the most important stepping stones in spectral asymptotics and provide to the interested reader the references to further work. Of course, we make no claim of completeness and refer to the cited literature for further work.
	
	The author would like to thank Alexander Strohmaier and Nikolas Eptaminitakis for the many discussions on this topic and the organisers of the workshop ``Analytical methods in interplay with physics'' at the University of Potsdam for their hospitality.

	\section{Sobolev spaces}\label{sect:sobolev spaces}

	\subsection*{Sobolev spaces on $\R^n$ and on closed manifolds}
	We recall the classical distribution and Sobolev spaces in order to fix our notation and conventions. Let $\Scal(\R^n)$ be the space of Schwartz functions, that is, $\phi\in\Scal(\R^n)$ if $\phi$ is smooth and rapidly decaying together with all its derivatives. It is a Fréchet space with seminorms
	\[
	p_k(\phi)=\sum_{|\alpha|+|\beta|\leq k}\sup_{x\in\R^n}|x^\alpha \del^\beta \phi(x)|.\] 
	The space of tempered distributions is the dual $\Scal'(\R^n)$ of $\Scal(\R^n)$, namely the space of all continuous functionals $u\colon\Scal(\R^n)\rightarrow \C$ which are continuous with respect to the above Fréchet space topology, namely $u\in \Scal'(\R^n)$ if
	\[|u(\phi)|\leq Cp_k(u)\quad \forall \phi\in\Scal(\R^n).\]
	It is equipped with the weak* topology: a sequence $(u_j)\subset\Scal'(\R^n)$ converges to 0 if and only if $\braket{u_j,\phi}\to 0$ for all $\phi\in\Scal(\R^n)$. 
	
	The space $\Cinf_c(\R^n)$ is the set of all smooth and compactly supported functions on $\R^n$ (we employ the same notation for the corresponding sheaf on a manifold $M$). It can be topologised via seminorms as well, but it is not a Fréchet space. This notwithstanding, it is usually known as the space of \textit{test functions} and its continuous dual $\Dcal'(\R^n)$ is the space of distributions on $\R^n$. Continuity is here taken to mean the following: for any compact set $K$ there are $C>0$ and $N\in\N$ such that for any $\phi\in\Cinf_c(\R^n)$ it holds
	\[
	|u(\phi)|\leq C\sum_{\alpha\leq N}\sup_{x\in K}|\del^\alpha \phi(x)|.
	\]
	
	The Fourier transform $\hat u=\Fcal(u)$ is defined, for $u\in\Scal(\R^n)$, by
	\[
	\hat{u}(\xi)=\Fcal(u)(\xi)=(2\pi)^{-n/2}\int_{\R^n} e^{-\im x\xi}u(x)dx\]
	and is again a Schwartz function. If we denote by $(\cdot,\cdot)$ the $L^2$-scalar product on $\R^n$, the Fourier transform has an adjoint $F^\ast$, given explicitly by
	\[
	\Fcal^\ast(u)(\xi)=(2\pi)^{-n/2}\int_{\R^n}e^{\im x\xi}u(x)dx,\]
	which is also the inverse of $\Fcal$. Thus, $\Fcal$ is an isomorphism on $\Scal(\R^n)$. It extends therefore, by duality, to an isomorphism of $\Scal'(\R^n)$ given by the formula
	\[
	\hat u(\phi)\equiv u(\hat \phi),\quad u\in\Scal'(\R^n),\phi\in\Scal(\R^n).\]
	Moreover, it also extends to an isometry of $L^2(\R^n)$.
	
	The Sobolev spaces $\sob^s(\R^n)$, $s\in\R$, are defined to consist of those $u\in\Scal'(\R^n)$ such that $\braket{\xi}^{-s}\hat{u}(\xi)\in L^2(\R^n)$. Here and throughout the manuscript, $\braket{z}^s=(1+|z|^2)^{-s/2}$ for any vector $z\in\R^n$ and $s\in\R$. If $s=k\in\N$, this definition is equivalent to the classical Sobolev spaces $W^{k,2}(\R^n)$, namely to requiring that $D^{\alpha}u\in L^2(\R^n)$ for all $\alpha\in\N^n$, $|\alpha|\leq k$.
	
	Elements of the spaces $\sob^s(\R^n)$ are not continuous functions in general, but a fundamental result implies that they embed in Hölder spaces $\Ccal^{k,\gamma}$ for high enough order:
	\begin{lemma}[Sobolev embeddings]
		\label{lemma:sobolev embedding}
		If $s=n/2+k+\gamma$ for $k\in\N$ and $\gamma\in(0,1)$, and if $u\in \sob^s(\R^n)$, then the equivalence class of $u$ contains an elements of $\Ccal^{k,\gamma}(\R^n)$. In other words, the identity map extends to a continuous embedding 
		\begin{equation*}
			\sob^{s}(\R^n)\hookrightarrow \Ccal^{k,\gamma}(\R^n).
		\end{equation*}
	\end{lemma}
	
	Let now $\Omega$ be a closed manifold, i.e. $\Omega$ is Hausdorff, second-countable, locally Euclidian, compact and without boundary. We can always choose a Riemannian metric on $\Omega$ and we assume to have done so. We define distributions on $\Omega$ to be elements of the dual space of $\Cinf_c(\Omega)$, as above (continuity is defined the same way, with the derivatives replaced by covariant differentiation with respect to the Levi-Civita connection). Define $\sob^{s}(\Omega)$ to consist of all $u\in \Dcal'(\Omega)$ for which, given a coordinate patch $U\subset \Omega$ with chart $\chi$ and a localising function $\psi\in\Cinf_c(U)$, the map $(\psi u)\circ\chi^{-1}$ is in $\sob^{s}(\chi(U))$. This condition is invariantly defined, as follows for example from the complex interpolation method (see \cite{taylor2023PDE1} for an explanation). Then, the results obtained for $\R^n$ carry over directly to this setting, and even a compactness result holds true.
	\begin{lemma}[Rellich compactness]
		\label{lemma:rellich compactness}
		If $\Omega$ is a closed manifold, $s\in \R$, $\sigma>0$, then the identity map extends to a compact embedding
		\[
		\sob^{s+\sigma}(\Omega)\to\sob^s(\Omega).\]
	\end{lemma}
	
	\subsection*{Sobolev spaces on manifolds with boundary.}
	We discuss now the spaces of distributions on
	\[
	\R^n_+=\{(x_1,x')\in\R^n\colon x_1>0,x'\in\R^{n-1}\}
	\]
	and its closure $\overline{\R^n_+}$, to prepare for the discussion on manifolds with boundary. For integer $k$ we want, as before, 
	\[
	\sob^{k}(\bar{\R}^n_+)=\{u\in L^2(\R^n_+)\colon \forall |\alpha|\leq k \;D^\alpha u\in L^2(\R^n_+)\}.
	\]
	One has that $\overline{\Scal}(\R^n_+)$, the space of restrictions to $\R^n_+$ of elements of $\Scal(\R^n)$, is dense in $\sob^{k}(\R^n_+)$ with respect to its natural Hilbert space topology. This is a consequence of the fact that the operators of translation in the $x_1$ direction are continuous in the $\sob^{k}$ topology. 
	
	One has that a distribution $u\in\sob^{k}(\R^n_+)$ is actually the restriction to $\R^n_+$ of $\tilde{u}\in\sob^{k}(\R^n)$. This follows from the existence of a continuous extension to $\sob^{k}(\R^n_+)$ of the map $E\colon
	{\Scal}(\R^n_+)\to \Scal(\R^n)$ given, for some coefficients $a_j\in\C$, by
	\[
	\tilde{u}(x)=\left\{\begin{aligned}
		&u(x),\quad x_1>0;\\
		&\sum_{j=1}^Na_ju(-jx_1,x'),\quad x_1<0.
	\end{aligned}\right.\]
	That the extension exists follows from an explicit computation of the $a_j$'s as solutions of a Vandermonde system. Such extension is then a right-inverse to the restriction map $\rho\colon\sob^{k}(\R^n)\to\sob^{k}(\R^n_+)$, which is therefore surjective. 
	
	We can now define the general Sobolev spaces $\sob^{s}(\R^n_+)$ as the complex interpolation spaces between $L^2(\R^n_+)$ and $\sob^{k}(\R^n_+)$ for some $k\geq s$. This is independent of the chosen $k$, since the complex interpolation method also provides an extension of $E$ to $\sob^{s}(\R^n_+)$, and gives
	\[
	\sob^{s}(\R^n_+)\cong\sob^{s}(\R^n)\diagup \{u\in \sob^{s}\colon u|_{\R^n_+}=0\}.
	\]
	With the notation we will adopt later, the above Sobolev spaces of \textit{extendible distributions} correspond to $\overline{\sob}_{s}({\R}^n_+)$, while the set in the quotient consists of \textit{supported distributions} and is denoted by $\dot{\sob}_{s}(\overline{\R}^n_-)$.
	
	Let now $\overline{M}$ be a compact smooth manifold with smooth boundary. We can always assume that $\overline{M}$ is a submanifold of a closed manifold of the same dimension, call it $\Omega$, by using a collar neighbourhood of $\del M$. For nonnegative integer $k$ we let $\bar\sob^k (M)$ be the set of all $u\in L^2(M)$ such that $Pu\in L^2(M)$ for any differential operator $P$ of order $k$ or less with coefficients in $\Cinf(\overline{M})$. Then, $\Cinf(\overline{M})$ is dense in $\bar \sob^k(M)$ and we get an extension operator $E\colon\bar\sob^{k}(M)\to\bar\sob^{k}(\Omega)$ by working in coordinate patches. Just like before we can define $\bar\sob^{s}(M)$ for $s\geq 0$ to be a complex interpolation space and obtain an isomorphism with
	\begin{equation}
		\label{eq:sobolev spaces on Omega are quotients}
		\bar\sob^s(M)\cong \sob^{s}(\Omega)\diagup\dot\sob^{s}(\overline{\Omega\setminus M}).
	\end{equation}
	Notice that the above characterisation can be used to define $\bar\sob^{-s}(M)$ as well, albeit it is not completely obvious that the spaces do not depend on the choice of the inclusion $M\subset \Omega$. 
	
	The Sobolev embeddings (Lemma \ref{lemma:sobolev embedding}) and the Rellich compactness (Lemma \ref{lemma:rellich compactness}) extend to this setting without change.
	
	Before introducing the next class of Sobolev spaces, let us recall the trace theorem. Denote by $\gamma$ the Dirichlet trace, namely $\gamma\colon \Cinf(\overline{M})\to\Cinf(\del M)$ is given by $\gamma u=u|_{\del M}$.
	\begin{lemma}
		\label{lemma:sobolev trace}
		If $s>1/2$, $\gamma$ extends uniquely to a continuous map 
		\[
		\gamma\colon\bar\sob^s(M)\to\sob^{s-1/2}(\del M).\]
	\end{lemma}
	The above is proven first in local coordinates by a direct computation and the carried over to $M$ via localisation and partitions of unity.
	
	We can now introduce the so-called ``Dirichlet-Sobolev'' spaces: For $s\geq 0$, $\sob^s_0(M)$ is the closure of $\Cinf_c(M)$ in $\bar \sob^s(M)$, thus its elements are obtained as limits in the $\sob^s(M)$ norm of sequences of smooth functions supported \textit{in the interior}.
	\begin{lemma}
		\label{lemma:dirichlet=supported}
		If $k$ is a nonnegative integer, then
		\[
		\sob^k_0(M)=\{u\in\sob^k(\Omega)\colon \supp u\subset\overline{M}\}\equiv \dot\sob^k(\overline{M}).\]
	\end{lemma}
	\begin{proof}
		Remark first that the topology on $\sob^k(M)$ is induced by the norm
		\[
		\|u\|^2_{\sob^k(M)}=\sum_{l=1}^N\|P_lu\|^2_{L^2(M)},\]
		where $N$ is an integer depending on $k$ and $M$ and the $P_l$'s are differential operators of order not bigger than $k$. This implies, by locality of the $P_l$'s, that the closure of $\Cinf_c(M)$ in $\sob^k(M)$ is the same as that in $\sob^k(\Omega)$. Then, by continuity, it is clear that any $\sob^k(\Omega)$ limit of sequences in $\Cinf_c(M)$ cannot be supported at points outside $\overline{M}$, so that $\sob^k_0(M)\subset \dot\sob^k(\overline{M})$.
		
		Conversely, let $u$ be an $\sob^k(\Omega)$ distribution with support in $\overline{M}$. If the support does not intersect the boundary, then the statement is promptly checked in local coordinates using the density result for open sets of $\R^n$. Assume therefore that the support intersects the boundary and let $x_0$ be a point in the intersection. Choose a diffeomorphism $\chi\colon U\to \overline{\R^n}$ for $U$ a small enough open neighbourhood of $x_0$ in $M$ and mapping $U\cap\overline{M}$ diffeomorphically to $\R^n_+$. Look then at $\tilde{u}\equiv u\circ\chi^{-1}$, which is an $\sob^k(\R^n)$ distribution supported in $\overline{\R^n_+}$. We can shift it by $-s$ in the $x_1$-direction to obtain a distribution with support fully in $\R^n_+$ and then we find a $v$ in $\Cinf_c(\R^n_+)$ being $\epsilon$-close to $\tau_{-s}\tilde{u}$ in the $\sob^k(\R^n)$ norm. By continuity of $\tau_s$ in the $\sob^k$-norms, we obtain 
		\[
		\|\tau_{-s}\tilde{u}-v\|_{\sob^k(\R^n)}<\epsilon \implies \|\tilde{u}-v\|_{\sob^k(\R^n)}\leq \epsilon,\]
		so that the pullback of $v$ will be $\epsilon$-close to $u$ in the chart $U$. By using a partition of unity and summing up, we obtain the global version.
	\end{proof}
	In fact, the equality of the spaces $\sob^{s}_0(M)$ and $\dot\sob^{s}(\overline{M})$ holds in much greater generality, for any $s\notin \N+\frac{1}{2}$, see the exercises for Section 4.5 in \cite{taylor2023PDE1}.
	
	Above, we mentioned how we can define $\sob^s(M)$ for negative $s$ as a quotient. For negative integers $k$, which is our case of interest, there is another possible definition, intrinsic to $\overline{M}$. It relies on the functional analytic characterisation of the dual space of a closed linear subspace $F$ of a Banach space $E$, namely
	\[
	F^\ast\cong E^\ast \diagup F^\perp,\]
	where $F^\perp$ is the annihilator of $F$ in $E^\ast$. In the above, take $E=\sob^k(\Omega)$ and $F=\sob^k_0(M)$, which is the closure of $\Cinf_c(M)$ in the $\sob^k(\Omega)$-norm. Then, we claim that the annihilator of $F$ is exactly $\dot\sob^{-k}(\overline{\Omega\setminus M})$. Indeed, if $u$ belongs to this last space, then it must vanish along any sequence $(\phi_j)\subset\Cinf_c(M)$ and, by continuity, on the limits too. Viceversa, if $u$ is a distribution in $\sob^{k}(\Omega)$, vanishing on every $\phi\in\bar \sob^{-k}(M)$, then no point $x\in M$ can belong to the support of $u$. For, if this were the case, we could find a $\phi\in\Cinf_c(M)$, $\phi(x)\neq 0$, and $u(\phi)\neq 0$.
	
	Putting together the quotient characterisation and the above argument, we obtain
	\begin{lemma}
		\label{lemma:dual of dirichlet-sobolev space}
		For any compact manifold $\Omega$ and open submanifold $M$ with smooth boundary and for any integer $k\geq 0$ we have a natural isomorphism
		\[
		\sob^k_0(M)^\ast\cong \bar\sob^{-k}(M).\]
	\end{lemma}
	
	Let $P$ a differential operator of order $2k$ written as 
	\[
	P=\sum_{l=1}^LA_lB_l
	\]
	for some differential operators $A_l,B_l$ or order $k$ with $\Cinf(\overline{M})$-coefficients. Then, $P$ defines a continuous operator $P\colon \sob^k_0(M)\to\bar\sob^{-k}(M)$ and for $u,v\in\sob^k_0(M)$ the pairing $\braket{u,Pv}$ makes sense, in view of the Lemma \ref{lemma:dual of dirichlet-sobolev space}. On the other hand, if $A^\dagger_l$ denotes the formal adjoint of $A_l$ on $\overline{M}$, then in fact
	\begin{equation}
		\label{eq:pairing of Hk0 and H-k is L2 scalar product}
	\braket{u,Pv}=\sum_{l=1}^L(A^\dagger_lu,B_lv)_{L^2(M)}
	\end{equation}
	and \eqref{eq:sobolev spaces on Omega are quotients} implies, by locality, that $P$ extends for each $s\in\R$ to a continuous map
	\[P\colon\bar\sob^s(M)\to\bar\sob^{s-2k}(M).\]
	In particular the case $s=k$ provides us with an extension of \eqref{eq:pairing of Hk0 and H-k is L2 scalar product} to $v\in \bar\sob^k(M)$, provided $u\in\sob^k_0(M)$.
	
	The last result in this section is going to play an important rôle later on. It is known in the literature as the \textit{Poincaré} or \textit{Poincaré-Wirtinger inequality}.
	\begin{theorem}\label{theo:poincare inequality}
		Suppose $\overline{M}$ is a compact, connected, smooth Riemannian manifold with (nonempty) boundary. There exists a constant $C=C(\bar M)>0$ such that for all $u\in\sob^1_0(M)$
		\begin{equation}
			\label{eq:poincare inequality}
			\|u\|^2_{L^2(M)}\leq C\|du\|^2_{L^2(M)}.
		\end{equation}
	\end{theorem}
	\begin{proof}
		By density, it suffices to prove the result for $u\in\Cinf_c(M)$. For $x\in M$ let $c_x$ be a smooth curve joining $x$ with a point $y$ on the boundary and write 
		\begin{equation*}
			u(x)=u(y)-\int_{c_x}du.
		\end{equation*}
		Since $y\in\del M$ and $u$ is compactly supported in the interior, $u(y)=0$. Picking $c_x$ to be the shortest geodesic connecting $x$ and $y$, we can estimate the absolute value of $u(x)$ by ($ds$ below is the measure on $c$)
		\begin{equation*}
			|u(x)|\leq\int_{c_x}\|du\|_{T^\ast_ x M}ds
		\end{equation*}
		and thus by Hölder's inequality
		\begin{equation*}
			|u(x)|\leq \sqrt{L(c_x)}\left(\int_{c_x}\|du\|^2_{T_x^\ast M}ds\right)^{1/2}.	
		\end{equation*}
		Squaring and integrating the above inequality over $M$, together with the obvious estimate $L(c_x)\leq \mathrm{diam}(M)$, $\mathrm{diam}(M)=\sup_{M\times M}\mathrm{dist}(x,y)$, gives the claim after one application of Fubini's theorem.
	\end{proof}	
	
	\section{The Laplace operator on a Riemannian manifold with boundary}\label{sect:laplace}
	
	In this section, $\overline{M}$ is a compact, connected, smooth Riemannian manifold with boundary $\del M$ and metric $g$. If needed, just like in the previous Section, $\overline{M}$ will be considered to be embedded in a closed manifold $\Omega$ as the closure of the open submanifold $M$. The Riemannian metric induces the Laplace-Beltrami operator $\Delta$, defined as the differential operator acting on $\Cinf_c(M)$ (adopting the Einstein convention on lowercase Latin indices)
	\begin{equation*}
		\Delta u=\frac{1}{\sqrt{\det g}}\partial_j\left(\sqrt{\det g}g^{jk}\partial_k u\right).
	\end{equation*}
	The above expression is coordinate-invariant and $\Delta$ can be expressed as an intrinsic differential operator using the Hodge $\ast$ induced by $g$. If $M$ has empty boundary, classical results imply that, when acting on $L^2$-based Sobolev spaces, $\Delta$ is Fredholm, has compact resolvent and thus discrete spectrum. Our purpose in this Section is to analyse one possible modification in the case $\del M\neq \emptyset$ in order to achieve discreteness of the spectrum. 	
	
	Let us remark, for starters, that the discussion below holds true for any second-order semi-bounded scalar differential operator, symmetric with respect to a smooth density $\gamma$ on $M$. However, for clarity (and since it is the most interesting case), we will fix the Riemannian density $\gamma_g$ and consider only the Laplace-Beltrami operator (we could add a smooth potential, but that would muddy the discussion). Moreover, in the interest of keeping the notation lighter, we will drop the overline from the Sobolev spaces of extendible distributions, that is
	\[
	\sob^k(M)\equiv \bar{\sob}^k(M).\]

	Thus, consider $\Delta$ as an unbounded operator on $L^2(M)$, acting on the dense domain $\Cinf_c(M)$. We start by investigating the solvability of $\Delta$ under the Dirichlet boundary condition, namely we look at the system for $f\in \Ccal(M)$
	\begin{equation}
		\label{eq:dirichlet problem}
		\left\{\begin{aligned}
			\Delta u&=f,\\
			u|_{\partial M}&=0.
		\end{aligned}\right.
	\end{equation}
	For $u\in\Cinf_c(M)$, we have by Green's identities 
	\begin{equation}\label{eq:green implies Laplace is positive}
		(\Delta u,u)=\|du\|^2_{L^2(M)}.
	\end{equation}
	On the other hand, Poincaré's inequality says that the $\sob^1(M)$-norm of $u$ and the $L^2$-norm of $du$ are equivalent on $\sob^1_0(M)$. Finally, the discussion preceding Theorem \ref{theo:poincare inequality} holds true for $\Delta$ and shows that it acts continuously
	\begin{equation}\label{eq:friedrichs extension of Laplace}
		\Delta\colon\sob^1_0(M)\to\sob^{-1}(M).
	\end{equation}
	By density, \eqref{eq:green implies Laplace is positive} holds true for $u\in\sob^1_0(M)$ and the Poincaré inequality implies the existence of a constant $C>0$ such that for all $u\in\sob^1_0(M)$
	\begin{equation*}
		(\Delta u,u)\ge C\|u\|^2_{\sob^1(M)}.
	\end{equation*}
	Bounding the LHS from above by $\|\Delta u\|_{\sob^{-1}(M)}\|u\|_{\sob^1(M)}$ shows further that
	\begin{equation}\label{eq:Laplace is bounded below}
		\|\Delta u\|_{\sob^{-1}(M)}\ge C\|u\|_{\sob^1(M)},
	\end{equation}
	so that $\Delta$ is bounded from below on $\sob^1_0(M)$. In particular, 0 is not an eigenvalue of $\Delta$ on $\sob^1_0(M)$.
	\begin{lemma}
		The extension \eqref{eq:friedrichs extension of Laplace} is bijective.
	\end{lemma}
	\begin{proof}
		The injectivity is a direct consequence of \eqref{eq:Laplace is bounded below}. Moreover, the range of $\Delta$ must be closed: if $(v_j)$ is a sequence in $\ran (\Delta)$ converging to $v\in \sob^{-1}(M)$, then the sequence $(u_j)\subset\sob^1_0(M)$ such that $\Delta u_j=v_j$ also converges to $u\in H^1_(M)$ (since the latter is a closed subspace of a Hilbert space and in view of \eqref{eq:Laplace is bounded below}) and then by continuity $\Delta u=v$. 
		
		Assume now $\Delta$ is not surjective, thus there must exist an element $v\in\sob^{-1}(M)^\ast$ such that $\braket{v,\Delta u}=0$ for all $u\in\sob^1_0(M)$. By Lemma \ref{lemma:dual of dirichlet-sobolev space}, however, $\sob^{-1}(M)^\ast\cong\sob^1_0(M)$, thus we can pick $u=v$ and see that $0=(\Delta v,v)=\|dv\|^2_{L^2(M)}$, so that $v=0$. This proves the lemma.
	\end{proof}
	Let $T\colon\sob^{-1}(M)\to\sob^1_0(M)$ be the bounded inverse to $\Delta$. Then, if $u_1,u_2\in\sob^1_0(M)$ with $v_j=\Delta u_j$, we can exploit \eqref{eq:green implies Laplace is positive} to obtain
	\begin{align*}
		\braket{Tv_1,v_2}&=\braket{T\Delta u_1,\Delta u_2}=\braket{u_1,\Delta u_2}\\
		&=(du_1,du_2)=\braket{\Delta u_1,u_2}\\
		&=(v_1,Tv_2).
	\end{align*}
	Thus, the restriction of $T$ to $L^2(M)$ is selfadjoint, and by the Rellich lemma it is also compact on $L^2(M)$. Directly from above, $T$ is also positive. It follows that $\Delta\colon\sob^1_0(M)\cap\sob^2(M)\to L^2(M)$ is bijective as well. This extension is called the \textit{Dirichlet extension}\footnote{This happens to coincide with the \textit{Friedrichs extension} of $\Delta$, see \cite{borthwick_2020_spectral_theory}, Subsection 6.1.2.} of the Laplacian, and will subsequently be denoted by $\Delta_D$.
	
	We prove now that the spectrum of the Dirichlet extension is discrete. We give two proofs of this fact, the first one easier after the preparations done until now, the second more in tune with our approach in Section \ref{sect:spectral function}:
	\begin{enumerate}
		\item Let $(u_j)$ be an orthonormal basis of eigenfunctions of $T$ with eigenvalues $\mu_j>0$, converging to 0. We have $u_j\in\sob^1_0(M)$ by definition, and furthermore 
		\[
		\Delta_D u_j=\Delta_D \frac{1}{\mu_j}T u_j=\frac{1}{\mu_j}u_j,
		\]
		so that $u_j$ are eigenfunctions of $\Delta_D$ with eigenvalues $\lambda_j=\frac{1}{\mu_j}$ growing to $\infty$. Since $(u_j)$ is complete, it diagonalises $\Delta_D$ as well and the spectrum is discrete.
		\item Recall that the spectral theorem associates with $\Delta_D$ a projection-valued measure $P\colon \Bfrak(\R)\to \Bcal(L^2(M))$, namely a map that assigns to each Borel set in $\R$ an orthogonal projection in $L^2(M)$, and that there is a bounded functional calculus for $\Delta_D$. This means that we can make sense of the operator $f(\Delta_D)$ for any bounded measurable function $f$, defining it for any Borel set $E$ by
		\begin{equation*}
			\label{eq:functional calculus}
			(\phi,f(\Delta_D)\phi)=\int_\R f(t)d(\phi,P_E\phi),
		\end{equation*}
		where the integral is taken with respect to the spectral measure $(\phi, P_E\phi)$ of $\Delta_D$. Recall that it is a positive measure on the Borel $\sigma$-algebra of $\R$.
		In particular, with each interval $I\subset\R$ we associate the operator $P_I=P(\chi_{I})$. Let $E_\lambda$ be the spectral family of $\Delta_D$, that is, the family of operators $P_{(-\infty,\lambda]}$ for $\lambda\in\R$. Clearly, $E_\lambda$ is supported in the positive real line since $\Delta_D$ is a positive operator (with empty kernel), furthermore $E_\lambda$ commutes with $\Delta_D$ (this is easily proven by splitting $L^2(M)$ orthogonally in $E_\lambda L^2\oplus (E_\lambda L^2)^\perp$ and looking at the action of $\Delta_D$ there).  We record the following inequality in a Lemma for later reference.
		\begin{lemma}
			\label{lemma:spectral estimate}
			Let $L^2(M)\ni u=E_\lambda u$. Then
			$\displaystyle{\|\Delta_D u\|\leq \lambda \|u\|}$.
		\end{lemma}
		\begin{proof}
			We split the proof in two steps:
		\begin{itemize}
			\item If $A$ is any positive selfadjoint operator and $E_\lambda$ is the spectral family of $A$, then $u=E_\lambda u$ is equivalent to saying that $(u,A u)\leq \lambda \|u\|^2$. Indeed, using the positivity of $A$ first and of the spectral measure $P_t$ second we have
			\begin{align*}
				(u,A u)&=(u,A E_\lambda u)=(u,E_\lambda A u)\\
				&=\int_\R t\chi_{(-\infty,\lambda]}(t)d(u,P_t u)\\
				&=\int_0^\lambda td(u,P_t u)\leq \lambda \int_0^\lambda d(u,P_tu)\\
				&\le\lambda\int_\R d(u,P_t u) =\lambda(u,u).
			\end{align*}
			All steps above are reversible since $A$ is selfadjoint and positive, proving the claimed equivalence.
			\item If $(u,\Delta_D u)\le \lambda\|u\|^2$, then we have
			\begin{equation*}
				\|\Delta_D u\|_{L^2}^2=(\Delta_D u,\Delta_D u)=(u,\Delta_D^2 u)\le \lambda^2\|u\|_{L^2}^2,
			\end{equation*}
			using the previous point with $A=\Delta_D^2$. This is the claimed inequality after taking the square root.
		\end{itemize}
		\end{proof}
		In particular, by the Poincaré inequality and Lemma \ref{lemma:spectral estimate}, $\|u\|_{\sob^1(M)} \leq C\lambda\|u\|_{L^2(M)}$. The set of such $u$'s with $\|u\|_{L^2(M)}\le1$ is compact in view of Rellich compactness, thus the range of $E_\lambda$ must be finite dimensional. Since this holds true for all generalised eigenfunctions, the spectrum is discrete and $L^2(M)$ decomposes as the the sum of finite-dimensional eigenspaces.
	\end{enumerate}
	
	To close this Section, we mention the following result establishing the higher regularity of solutions to the Dirichlet problem.
	\begin{theorem}[Higher regularity]
		In \eqref{eq:dirichlet problem}, assume the source $f\in\sob^{k-1}(M)$ for some $k\in\N$. Then, if $u\in\sob^1_0(M)$ solves \eqref{eq:dirichlet problem}, we have $u\in\sob^{k+1}(M)$ and for all $u\in\sob^{k+1}(M)\cap\sob^{1}_0(M)$ we have the \textit{elliptic regularity} estimate
		\begin{equation*}
			\|u\|^2_{\sob^{k+1}}\lesssim \|\Delta_D u\|^2_{\sob^{k-1}}+\|u\|^2_{\sob^{k}}.
		\end{equation*}
		In particular, if $u$ is an eigenfunction, then $u\in\Cinf(\overline{M})$.
	\end{theorem}
	As a consequence, we observe that we can equivalently characterise the norm of $\sob^{k}(M)$ using the functional calculus for $\Delta_D$.
	\begin{corollary}
		\label{cor:equivalent sobolev norm via laplace}
		Let $u\in\Dcal'(M)$, $k$ a positive integer. Then $u\in\sob^{k}(M)$ if, and only if, $\Delta^k_D u\in L^2(M)$.
	\end{corollary}

\section{The spectral function}\label{sect:spectral function}
	Let us denote by $E_\lambda$, as before, the spectral family of $\Delta_D$. If $(u_j)$ is an orthonormal basis of eigenfunctions of $L^2(M)$, with eigenvalues $0< \lambda_1\leq \lambda_2\leq\dots$ (repeated according to multiplicity), we can write the kernel of $E_\lambda$ as
	\begin{equation}
		\label{eq:kernel of spectral family}
		K_{E_\lambda}(x,y;\lambda)=\mu_g(y)\sum_{j\colon\lambda_j\leq\lambda}\bar u_j(x)u_j(y)\equiv \mu_g(y)e(x,y;\lambda),
	\end{equation}
	where $e(x,y;\lambda)$ is the \textit{spectral function} of $\Delta_D$. We also denote by $N(\lambda)$ the \textit{eigenvalue counting function}, namely $N(\lambda)=\sharp\{j\colon \lambda_j\leq \lambda\}$. 
		
	In order to obtain estimates for $e$ and other objects that depend on $\lambda$, we recall the following \textit{parametric} Sobolev lemma (cf. \cite{hormander1994analysispseudodifferential}, Lemma 17.5.2).
	\begin{lemma}
		\label{lemma:parametric sobolev embedding}
		Let $u\in \sob^k(M)$ and $l<k-\frac{n}{2}$. Then $u\in\Ccal^l(\overline{M})$ and for $\lambda\geq 1$
		\begin{equation}
			\label{eq:parametric sobolev inequality}
			\lambda^{k-\frac{n}{2}-l}\sum_{|\alpha|\leq l}\sup|D^\alpha u|^2\leq C(\|u\|^2_{\sob^k(M)}+\lambda^k\|u\|^2_{L^2}).
		\end{equation}
	\end{lemma}
	As a direct application, we can estimate the spectral function.
	\begin{theorem}\label{theo:direct estimate of spectral function}
		For each $\alpha\in\N^{2n}$ there is a positive constant $C_{\alpha}$ such that, uniformly in $x,y\in \bar{M}$,
		\begin{equation*}\label{eq:direct estimate of spectral function}
			|D^\alpha_{x,y}e(x,y;\lambda)|\leq C_\alpha\lambda^{\frac{n+|\alpha|}{2}}.
		\end{equation*}
	\end{theorem}
	\begin{proof}
		We begin by showing that, whatever the $f\in L^2(M)$ and the $\alpha\in\N^n$, there are constants $C_\alpha>0$ such that $D^\alpha E_\lambda f$ is continuous and $|D^\alpha E_\lambda f(x)|\leq C_\alpha \lambda^{|\alpha|+n/2}\|f\|_{L^2}$.
		
		Indeed, let $u=E_\lambda f$. Then, $u$ is the linear combination of finitely many eigenfunctions of $\Delta_D$, so it's smooth and actually in $\cap_{k}\sob^k$ by Lemma \ref{lemma:parametric sobolev embedding}. In particular, for every positive integer $k$ we can first use Corollary \ref{cor:equivalent sobolev norm via laplace} and then, repeatedly, Lemma \ref{lemma:spectral estimate} to obtain constants $C_k$ such that
		\begin{equation*}
			\label{eq:multiple spectral estimate}
			\|u\|_{H^{2k}}\leq C_k\|\Delta_D^ku\|_{L^2}\leq C_k\lambda ^k\|f\|_{L^2}.
		\end{equation*}
		Now, for a fixed $\alpha\in\N^n$, we can pick $k$ such that $2k>|\alpha|+n/2$ and use Lemma \ref{lemma:parametric sobolev embedding} to discover that $D^\alpha u\in\Ccal(M)$ with the pointwise bound
		\begin{equation*}\label{eq:L^infinity bound for derivatives of E_lambda f}
			|D^\alpha E_\lambda f(x)|^2\leq C_\alpha \lambda^{|\alpha|}\|u\|_{H^{2k}}^2\leq C_\alpha \lambda^{|\alpha|+n/2}\|f\|^2_{L^2}.
		\end{equation*}
		On the other hand, let $h_\alpha(y;\lambda)\equiv D^\alpha_xe(x,y;\lambda)$. Then, integration by parts, together with the fact that eigenfunctions of the Dirichlet problem vanish at the boundary, shows that
		\[\begin{aligned}
		D^\alpha E_\lambda f(x)&=(\bar f,h_\alpha)_g,\\
		D^\alpha E_\lambda^2 f(x)&=(E_\lambda\bar f,h_\alpha)_g.
		\end{aligned}\]
		However, recall that $E_\lambda$ is a selfadjoint projection. Thus, the LHS are equal and we obtain that 
		\[
		(\bar f,E_\lambda h_\alpha)=(\bar{f}, h_\alpha)
		\] for any $f\in L^2(M)$. It follows $E_\lambda h_\alpha=h_\alpha$ with the bound $\|h_\alpha\|^2\le C_\alpha\lambda^{|\alpha|+n/2}.$ Using that $\|\Delta_D^kh_\alpha\|_{L^2}\le \lambda^k\|h_\alpha\|_{L^2}$, in conjunction with \eqref{eq:parametric sobolev inequality}, gives now
		\begin{equation*}
			|D^\beta_y h_\alpha(y)|^2\le C_{\alpha\beta}\lambda^{|\alpha|+|\beta|+n},
		\end{equation*}
		which is the claim after taking the square root.
	\end{proof}
	As an immediate corollary, we obtain the coarse asymptotic behaviour of the eigenvalue counting function. Indeed, 
	\begin{equation}\label{eq:coarse egv asymptotics}
		N(\lambda)=\Tr E_\lambda=\int_M e(x,x;\lambda)dx=C_0\mathrm{vol}(M)\lambda^{n/2}
	\end{equation}
	is obtained by definition of $N$ and $e$ and in view of the estimate \eqref{eq:direct estimate of spectral function} for $\alpha=0$. 
	
	The goal for the rest of this Section is to prove some results for the \textit{cosine transform} of the spectral measure,
	\begin{equation}\label{eq:cosine transform}
		\cos(t\sqrt{\Delta_D})=\int_0^\infty \cos(t\sqrt{\lambda})dE_\lambda. 
	\end{equation}
	\begin{lemma}
		The distributional kernel $K(t,x,y)\in\Dcal'(\R\times M\times M)$ of $\cos(t\sqrt{\Delta_D})$ is $\Fcal_{\tau\to t}(dm)$, where $m$ is the temperate measure
		\[
		m(x,y,\tau)=\frac{1}{2}\mu_g(y)\sgn(\tau)e(x,y,\tau^2).\]
	\end{lemma}
	\begin{proof}
		It suffices to prove this in local coordinates. Taking $\psi\in\Scal(\R)$ with $\hat{\psi}\in\Cinf_c(\R^n)$ and $f\in\Cinf_c(M)$, denote
		\[
		e(f,f;\lambda)=(E_\lambda f,f)_g=\int_{M\times M}e(x,y,\lambda)\bar{f(x)}f(y)\mu_g(x)\mu_g(y)dxdy,
		\]
		which is an increasing function of $\lambda$, bounded by $C\|f\|^2_{L^2}$ for a positive constant $C$. Then
		\begin{equation*}
			\int_\R (\cos(t\sqrt{\Delta_D}f,f))\psi(t)dt	=\int_\R \left(\psi(t)\int_M \cos(t\sqrt\lambda)de(f,f;\lambda)\right)dt,
		\end{equation*}
		and we are allowed to interchange the order of integration to see that
		\begin{align*}
			\int_\R(\cos(t&\sqrt{\Delta_D}f,f)_g)\psi(t)dt\\&=\frac{1}{2}\int_\R(\hat{\psi}(\sqrt\lambda)+\hat{\psi}(-\sqrt{\lambda}))de(f,f;\lambda)\\
			&=\frac{1}{2}\int_0^\infty(\hat{\psi}(\tau)+\hat\psi(-\tau))de(f,f;\tau^2)\\
			&=\frac{1}{2}\int_{M\times M}\bar{f(x)}f(y)\mu_g(x)\mu_g(y)\left(\int_0^\infty (\hat\psi(\tau)+\hat\psi(-\tau))d_\tau e(x,y;\tau^2)\right)dxdy.
		\end{align*}
		Thus we have found that, formally, $K(t,x,y)$ is the Fourier transform with respect to $\tau$ of the measure
		\[m(x,y;\tau)=\frac{1}{2}\mu_g(y)\sgn (\tau) e(x,y;\tau^2).\]
		In order to conclude, it suffices to show that the above measure is a temperate distribution, so that the Fourier transform is well-defined.
		
		By polarisation, we see that for any $a,b\in\C$ the function
		\begin{equation*}
			|a|^2e(x,x;\lambda)+\bar{a}be(x,y;\lambda)+a\bar be(y,x;\lambda)+|b|^2e(y,y;\lambda)
		\end{equation*}
		is increasing in $\lambda$. On the other hand, it is bounded by $C\lambda^{n/2}$ in view of Theorem \ref{theo:direct estimate of spectral function}. Thus, it is temperate and the proof is complete.
	\end{proof}
	
	The cosine transform \eqref{eq:cosine transform} is the solution operator of the time-independent wave equation on $\bar{M}$ with given initial data for $u$ and vanishing data for $\dot u$ and $u|_{\partial M}$. Indeed, given $f\in\Cinf_c(M)$, we have that $u(t,x)=\cos(t\sqrt{\Delta_D})f(x)$ satisfies $u(0,x)=f(x)$, $\dot u(0,x)=0$. Moreover, for arbitrary $k,l\in\N$
	\[
	\|\partial_t^k \Delta_D^l u\|_{L^2}\leq \|\Delta_D^{l+k/2}f\|_{L^2},
	\]
	which implies, thanks to Lemma \ref{lemma:parametric sobolev embedding} and the regularity results as the end of Section \ref{sect:laplace}, that $u\in\Cinf(\R\times \bar{M})$ and $u|_{\R\times\del M}=0$. By definition, it also holds true that 
	\[\left(\frac{\partial^2}{\del t^2}+\Delta_D\right)u(t,x)=0.\]
	Thus, we can recover $u$ approximately in the interior by solving the wave equation for $\Delta_D$ with initial data $u|_{t=0}=f$ and $\dot u|_{t=0}=0$, using the \textit{Hadamard parametrix}.
	
	\section{The Hadamard parametrix}
	
	The Hadamard parametrix construction is a general method of constructing an approximate solution to second order operators whose principal symbol is given by a symmetric nondegenerate 2-tensor. We sketch here how to do it in case of the wave equation.
	
	First, it relies on the existence of a certain family of (homogeneous) distributions $R_\nu$ on Minkowski space, called Riesz distributions. They are defined by the oscillatory integral
	\begin{equation}\label{eq:riesz distributions}
		R_\nu(t,x)=\frac{\nu !}{(2\pi)^{n+1}}\int_{\Im\tau=c<0}e^{\im (x\xi+t\tau)}(|\xi|^2-\tau^2)^{-\nu-1}d\xi d\tau.
	\end{equation}
	We collect their properties in the following Lemma, where $\Delta_D$ is as before the Laplace-Beltrami operator of the Euclidian metric (thus positive!). Beware that we use the notation
	\[
	\chi_+(s)=\left\{\begin{aligned}
		s,\quad s>0,\\
		0,\quad s\le 0,
	\end{aligned}\right.
	\] 
	and write $\check{R}_\nu$ for the distribution defined by the same formula with $(t,x)$ changed to $(-t,-x)$. 
	\begin{lemma}
		\label{lemma:properties of riesz distributions}
		The distributions $R_\nu$ satisfy:
		\begin{enumerate}
			\item $R_\nu$ is homogeneous of degree $2\nu+1-n$ and supported in $J^+(0)$;
			\item $\displaystyle{R_\nu=C(\nu,n)\chi_+^{\nu+\frac{1-n}{2}}(t^2-|x|^2)}$ for all $t>0$ and some positive constant $C(\nu,n)$, so $R_\nu$ only depends on $x$ through its absolute value;
			\item $(\del_t^2+\Delta_D)R_{0}=\delta_0$;
			\item $(\del_t^2+\Delta_D)R_{\nu}=\nu R_{\nu-1}$ and $-2\nabla_xR_\nu=xR_{\nu-1}$ if $\nu>0$;
			\item for $t\ge 0$, $R_\nu$ is a smooth function of $t$ with values in $\Dcal'(\R^n)$ and satisfies
			\[
			\lim_{t\to 0}\del^k_t R_{\nu}(t,x)=0 \text{ for }k\le 2\nu,\quad \lim_{t\to 0}\del_t^{2\nu+1}R_{\nu}(t,x)=\nu!\delta(x);
			\]
			\item the difference $R_\nu-\check{R}_\nu$ satisfies \[
			WF(R_\nu-\check{R}_\nu)=\{(t,x,\tau,\xi)\colon t^2=|x|^2, \tau^2=|\xi|^2, \tau x+t\xi=0\};\]
			\item $R_\nu-\check{R}_\nu$ and its time derivative are continuous functions of $x$ with values in $\Dcal'^{2k}(\R)$, provided $k$ is an integer with $k\ge \frac{n-1}{2}-\nu$. In the extreme case $k=\frac{n-1}{2}-\nu$, we have for $x=0$
			\[
			\del_t(R_\nu-\check{R}_{\nu})=C(\nu,n)\delta^{(2k)}(t);\]
			\item $2\del_t(R_0-\check{R}_0)$ is the Fourier transform of $de_0(x,\tau^2)$ where
			\[
			e_0(x,\tau^2)=(2\pi)^{-n}\int_{|\xi|<|\tau|}e^{\im x\xi}d\xi.\]
		\end{enumerate}
	\end{lemma}
	To emphasize 2. above, one usually writes $R_\nu(t,|x|)$ and we abide by this convention. Notice that, usually, Riesz distributions are introduced via 2., namely as a power of the Lorentzian distance function (supported in the positive light cone).
	
	Second, the distributions $R_\nu$ can be used to obtain a parametrix for the wave operator on Minkowski space. Indeed, with the initial condition $u_{-1}=0$, one can iteratively solve the transport equations
	\begin{equation*}
		\label{eq:transport equations}
		2\nu u_\nu+2\braket{x,\del_x u_\nu}+2\Delta_D u_{\nu-1}=0
	\end{equation*} 
	to determine a sequence of functions $\{u_\nu\}$. Taken as coefficients (the so-called \textit{Hadamard coefficients}), they produce the Hadamard parametrix of order $N$, $\sum_{\nu=0}^Nu_\nu(x)R_\nu(t,|x|)$. That this is an approximate inverse is a direct computation, using geodesic coordinates in a convex neighbourhood $V$ of 0 and the properties of Lemma \ref{lemma:properties of riesz distributions}. Indeed
	\begin{equation*}
		\left(\frac{\del^2}{\del t^2}+\Delta_D\right)\sum_{\nu=0}^Nu_\nu(x)R_\nu(t,|x|)=(\sqrt{|g|})\delta(t,x)+(\Delta_D u_N(x))R_N(t,|x|),
	\end{equation*}
	with the last error term being $\Ccal^k$ provided $N>k+\frac{n-1}{2}$. Moreover, given $c>0$ such that $B_c(0)\subset V$, one has $R_\nu(t,x)=0$ in a neighbourhood on $\{x\colon |x|\ge c\}$, provided $t<c$. 
	
	Third, the construction in flat space can be extended to Riemannian manifolds. Here and in what follows, $s(x,y)$ is the geodesic distance between the points $x,y\in\bar{M}$.
	\begin{theorem}[Hadamard parametrix]
		\label{theo:hadamard parametrix}
		Let $\bar M$ be compact Riemannian manifold with boundary, $V$ open with $\bar V\subset M$. There exist $c>0$ and functions $U_\nu\in\Cinf(M\times V)$ such that for all $(t,x,y)\in(-\infty,c)\times M\times V$ with $s(x,y)\le c$ it holds true
		\begin{equation}\label{eq:hadamard parametrix}
				\left(\frac{\del^2}{\del t^2}+\Delta_D\right) \sum_{\nu=0}^NU_\nu(x,y)R_\nu(t,s(x,y))= (\sqrt{|g|})\delta_{0,y}+(\Delta_D U_N(x,y))R_N(t,s(x,y)).
		\end{equation}
		The coefficients $U_\nu$ are the Hadamard coefficients of $V$, and are obtained by integrating the transport equations \eqref{eq:transport equations} on $M$ in geodesic coordinates near the diagonal. The remainder term is in $\Ccal^k((-\infty,c)\times M\times V)$, provided $N>k+ \frac{n-1}{2}$.
	\end{theorem}
	Remark that the condition that the geodesic distance be smaller than $c$ ensures that we do not reach the boundary of $M$ in time smaller than $c$. Thus, the above result does not provide a parametrix for the mixed Dirichlet-Cauchy problem. It is in fact possible to modify the construction of the Hadamard parametrix to include the Dirichlet boundary condition, see the discussion in \cite{hormander1994analysispseudodifferential} leading to Proposition 17.4.4.
	
	This concludes our overview of the construction of the Hadamard parametrix. We will now show, to conclude this Section, that it can be used to approximate the Fourier transform of the kernel of the cosine transform of $\Delta_D$. We begin by stating an analytical lemma that controls the error in the approximation, cf. \cite{hormander1994analysispseudodifferential}, Lemma 17.5.4 for a proof.
	\begin{lemma}
		\label{lemma:error estimate}
		Let $k\in\N$ and let $h\in\Cinf(\R\times\bar{M})$ satisfy $\del_t^l h=0$ at $t=0$ for all $l<k$. Assume furthermore that $v\in\Cinf([0,T]\times \bar{M})$ solves
		\begin{equation*}\label{eq:mixed problem with vanishing Cauchy data}
			\begin{aligned}
				\left(\frac{\del^2}{\del t^2}+\Delta_D\right)v&=h\quad\text{on }[0,T]\times \overline{M},\\
				v&=0\quad \text{on }[0,T]\times \del M,\\
				v=\dot v&=0\quad \text{at }t=0.
			\end{aligned}
		\end{equation*}
		Then
		\begin{equation}\label{eq:error estimte for the mixed problem}
			\sum_{l=0}^{k+1}\|D_t^{k+1-l}v(t,\cdot)\|_{\sob^l}\lesssim \int_0^t\|D^k_s h(s,\cdot)\|_{L^2}ds+\sum_{l=0}^{k-1}\|D^{k-1-l}_th(t,\cdot)\|_{\sob^l}.
		\end{equation}
	\end{lemma}
	In the statement and proof below, we let $s(x)$ be the geodesic distance from $x$ to $\del M$ and denote by $M_\rho=\{x\in M\colon s(x)>\rho\}$ for some $\rho>0$.
	\begin{theorem}
		\label{theo:hadamard approximation of spectral measure}
		Let $\bar{M}$ be a compact Riemannian manifold with boundary and choose $d>0$ such that \eqref{eq:hadamard parametrix} holds true for all $(x,y)\in M\times M_\rho$ with $s(x,y)<d$ and $\rho<d$. Moreover let $\Omega=\{(t,x)\in\R\times M\colon |t|<\min(s(x),d)\}$. Then
		\begin{equation}\label{eq:approxmation of spectral measure via hadamard}
			\widehat{dm}(x,x,t)-\sum_{2\nu< n}\del_t\left(R_\nu(t,0)-\check{R}(t,0)\right)U_{\nu}(x,x)|g(x)|^{{\frac{1}{2}}}
		\end{equation}
		is in $|t|^{n\!\!\!\mod 2}\Cinf(\Omega)$, with all derivatives bounded in $\Omega$. Whatever the dimension, the Taylor expansion of \eqref{eq:approxmation of spectral measure via hadamard} with respect to $t \ge 0$ is 
		\begin{equation}\label{eq:taylor expansion of approximation}
			\sum_{2\nu\ge n} \del_t\left(R_{\nu}(t,0)-\check{R}_\nu(t,0)\right)U_\nu(x,x)|g(x)|^{\frac{1}{2}}.
		\end{equation}
	\end{theorem}
	\begin{proof}
		Let $y\in M_\rho$ and $t<\rho$. Then the Hadamard parametrix 
		\begin{equation*}
			\label{eq:hadamard parametrix 2}
			\Ecal(t,x,y)\equiv\sum_{\nu=0}^NU_\nu(x,y)R_{\nu}(t,s(x,y))
		\end{equation*}
		is defined for $x\in M$, with $\Ecal=0$ near $\del M$. If we take $f\in\Cinf_c(M_\rho)$, the function $u(t,x)$ defined by
		\begin{equation*}
			\label{eq:solution to wave equation with compact data}
			u(t,x)\equiv\int \Ecal(t,x,y)|g(y)|^{1/2}f(y)dy
		\end{equation*}
		is in ${\Cinf}([0,\rho]\times M)$ and satisfies $u=0$ on $[0,\rho]\times M$ and $u(0,x)=0$, $\dot u(0,x)=f(x)$.
		
		The assignment $f\mapsto u$ is continuous in the $\Cinf$-topology. Indeed, changing variable in the integral to $z$ such that $\exp_y z=x$, where $\exp$ is the exponential map of $M$, we have $z=s(x,y)$. Thus, using \textasciitilde\; to signify that the function has been rewritten using the new coordinate $z$, we are led to 
		\begin{equation*}
			u(t,x)=\sum_{\nu=0}^N\int \tilde U_\nu(x,z)R_\nu(t,z)|\tilde g(z)|^{1/2}\tilde f(z)dz.
		\end{equation*}
		In this expression, the distributions $R_\nu$ are acting on smooth functions of $z$, depending smoothly on the parameter $x$. In particular, the dependence is continuous as claimed.
		
		It follows that
		\begin{equation*}
			v(t,x)=\cos\left(t\sqrt{\Delta_D}\right)f(x)-\frac{\partial u}{\partial t}(t,x)
		\end{equation*}
		has vanishing Cauchy and Dirichlet data. Furthermore, by \eqref{eq:hadamard parametrix}, $v$ solves the wave equation approximately. Letting $r_N(t,x,y)$ denote
		\begin{equation*}
			r_N(t,x,y)=(\Delta_D U_N(x,y))R_{N}(t,s(x,y)),
		\end{equation*}
		we have indeed
		\begin{equation*}
			\left(\frac{\del ^2}{\del t^2}+\Delta_D\right)v(t,x)=-\int\frac{\del r_N}{\del t}|g(y)|^{1/2}f(y)dy
		\end{equation*}
		where $r_N\in\Ccal^{k+1}$ for $N>k+(n+1)/2$ in view of Lemma \ref{lemma:properties of riesz distributions}. Using Lemma \ref{lemma:error estimate}, one obtains that all derivatives of $v$ are bounded by a power of $t$ times the $L^1$-norm of $f$.
		
		Now, for $t>0$ set
		\begin{equation*}
			K_N(t,x,y)=\widehat{dm}(x,y,t)-\partial_t \Ecal(t,x,y)|g(y)|^{1/2}.
		\end{equation*}
		This is continuous in $t$ with values in $\Dcal'(M\times M_0)$. Then, the same argument as above shows that $K_N\in\Ccal^{N-n-3}$ and all its derivatives are bounded by a power of $t$, 
		\begin{equation*}
			\left|D^\alpha_{x,y,t}K_N(t,x,y)\right|\le Ct^{2N-|\alpha|-n}, \quad |\alpha|\le N-n-3.
		\end{equation*}
		Since $\widehat{dm}$ is even in $t$, the same bounds hold true for $t\in\R$, provided we replace $K_N$ with
		\begin{equation*}
			\widehat{dm}(x,y,t)-\del_t\left(\Ecal(t,x,y)-\Ecal(-t,x,y)\right)|g|^{1/2}.
		\end{equation*}
		In view of Lemma \ref{lemma:properties of riesz distributions} we have that the above is a continuous function of $(x,y)$ with values in $\Dcal'(\R)$. Restriction to the diagonal gives the claim and finishes the proof.
	\end{proof}
	
	\section{The Tauberian theorem and the Weyl law}
	
	We are finally ready to discuss the proof of the Weyl law. It is based on a result of \textit{Tauberian} type, Theorem \ref{theo:tauberian theorem} below. Let us remark, for starters, that we have amassed by now a substantial amount of information on the cosine transform of the spectral measure. In particular, we have seen that it leads to an asymptotic solution to the wave equation. Moreover, 8. in Lemma \ref{lemma:properties of riesz distributions} shows that the first Hadamard coefficient, computed at the origin, has a geometric meaning, being related to the volume of the Euclidian ball. The Fourier Tauberian theorem below (cf. \cite{hormander1994analysispseudodifferential}, Lemma 17.5.4) allows us to translate the control we have on the cosine transform into explicit bounds for the spectral function.
	
	\begin{theorem}[Fourier Tauberian Theorem]
		\label{theo:tauberian theorem}
		Let $f$ be an increasing temperate function and $g$ a function of locally bounded variation with $f(0)=g(0)=0$. Assume further that there exist $p\in[0,n-1]$ and positive constants $a, c_1,c_2, M_1,M_2$ with $c_j\ge a$ such that
		\begin{equation}
			\label{eq:tauberian assumptions}
			\begin{aligned}
			|dg(\tau)|&\le M_1(|\tau|+c_1)^{n-1}d\tau\\
			|(df-dg)\ast\phi_a(\tau)|&\le M_2(|\tau|+c_2)^p,\quad \forall \tau\in\R.
		\end{aligned}
		\end{equation}
		Then, there exists a constant $C>0$ depending on $p$ and $n$ only such that
		\begin{equation}
			\label{eq:tauberian estimate}
			|f(\tau)-g(\tau)|\le C\left(M_1a\left(|\tau|+c_1\right)^{n-1} +M_2\left(|\tau|+a\right)\left(|\tau|+c_2\right)^p\right).
		\end{equation}
	\end{theorem}
		
	\begin{theorem}[Local Weyl law]
		\label{theo:local weyl law}
		There exist a positive constant $C$ such that the spectral function satisfies 
		\begin{equation}\label{eq:local weyl law}
			\left|e(x,x;\lambda)|g(x)|-e_0(0,\lambda)\sqrt{|g(x)|}\right|\le C\frac{\lambda^{n/2}}{1+s(x)\sqrt{\lambda}}.
		\end{equation}
	\end{theorem}
	\begin{proof}
	In the region $s(x)\sqrt{\lambda}\le 1$, the claim is just a consequence of Theorem \ref{theo:direct estimate of spectral function} for $\alpha=0$. We thus focus on the region $s(x)\sqrt{\lambda}>1$, where we shall apply Theorem \ref{theo:tauberian theorem} to the functions
	\begin{align*}
		f(\tau)&=m(x,x;\tau)=|g(x)|\sgn(\tau) e(x,x;\tau^2)/2\\
		g(\tau)&=\sgn(\tau) e_0(0,\tau^2)\sqrt{|g(x)|}/2,		
	\end{align*}
	with $a=\frac{1}{\min(s(x),d)}$. Here, $d$ is the same number that one chooses to construct the Hadamard parametrix in Theorem \ref{theo:hadamard approximation of spectral measure}.
	
	With the exception of the second estimate \eqref{eq:tauberian assumptions}, the assumptions of $f$ and $g$ are clearly satisfied. On the other hand, Theorem \ref{theo:hadamard approximation of spectral measure}, together with Lemma \ref{lemma:properties of riesz distributions}, gives that the leading term in the asymptotic expansion of $\widehat{df}$ in terms of homogeneous distributions is exactly $\widehat{db}$. More precisely, we have smooth functions $\nu_l(x)$ such that, in the sense of distributions,
	\begin{equation*}\label{eq:approximation up to homogeneous distributions}
		(\widehat{df}-\widehat{dg})(t)=\sum_{l=1}^{\frac{n-1}{2}}\Fcal_{\tau\to t}(|\tau|^{n-1-2l})(t)\nu_l(x).
	\end{equation*}
	Recall now that the Fourier transform of a homogeneous distribution on $\R^n$ of degree $\alpha$ is again homogeneous, of degree $-\alpha-n$. It follows
	\begin{equation*}
		\label{eq:approximation up to homogeneous distributions 2}
		\begin{aligned}
			\left(df-dg\right)\ast\phi_a(\tau)&=(\Fcal_{\tau\to t})^{-1}\left(\sum_{l=1}^{\frac{n-1}{2}}\Fcal_{\tau\to t}(|\tau|^{n-1-2l})\hat{\phi}_a(t)\nu_l(x)\right)\\
			&=\sum_{l=1}^{\frac{n-1}{2}}(|\cdot|^{n-1-2l}\ast\phi_a)(\tau)\nu_l(x).
		\end{aligned}
	\end{equation*}
	For $l=(n-1)/2$, the convolution is just the integral of $\phi_a$, that is, 1. The other terms are just regularisations of $|\cdot|^{k}$ near zero, for $k$ a positive integer smaller than $n$. Thus, $(df-dg)\ast\phi_a$ is the sum of a bounded function, $\nu_{(n-1)/2}$, and regularisations of $|\cdot|^k$ multiplied by $\nu_l$. The highest order power that can appear is either $0$ for $n\le 3$ or $n-3$. It follows that the estimates \eqref{eq:tauberian assumptions} hold true with $p=\max(n-3,0)$. Consequently, the estimate \eqref{eq:tauberian estimate} holds true with the same $p$, so that the second summand is in fact of lower degree compared to $(|\tau|+c_1)^{n-1}$, as $\tau\to \infty$.
	
	Therefore, in the region $s(x)\sqrt{\lambda}>1$, we obtain the claim after taking $\tau=\sqrt{\lambda}$ and rearranging the estimate. The proof is complete.
	
	\end{proof}
	
	\begin{corollary}[Weyl law]\label{cor:weyl asymptotics}
		The function $N(\lambda)$ has the asymptotic behaviour
		\begin{equation}
			N(\lambda)\sim C_n\mathrm{vol}(M)\lambda^{n/2}+O(\lambda^{(n-1)/2}\log\lambda),
		\end{equation}
		where $C_n$ is $(2\pi)^{-n}$ times the volume of an Euclidian unit ball.
	\end{corollary}
	\begin{proof}
		The estimate is obtained by integrating \eqref{eq:local weyl law} over $M$ with respect to the Riemannian volume form. The precise value of the constant follows from Lemma \ref{lemma:properties of riesz distributions}.
	\end{proof}
	
	Remark that this error bound is not the optimal one: it is known that one can always achieve $\lambda^{(n-1)/2}$, and in general (that is, without more restrictive geometric assumptions) this cannot be improved.
	
	\section{Historical notes and outlook}
	The asymptotic formula for the eigenvalue distribution of the Laplacian was first proven by \textcite{weyl_asymptotische_1912} in 1912 for a bounded planar domain with smooth boundary. It used the technique of \textit{Dirichlet-Neumann bracketing}: the interior of the domain is decomposed in squares of a given side and at each interface one solves a Dirichlet problem on one side and a Neumann problem on the other, and estimates the contribution of each of these to the counting function. Generalisations of this approach to higher dimensions appeared in subsequent years. 
	
	According to multiple sources\footnote{We were not able to verify this.}, Carleman devised a different method, closer to ours above: one can obtain information on the spectral function by studying the \textit{resolvent kernel} and then bound the counting functions by applying a Tauberian theorem. Inspired by this approach, in 1949 \textcite{minakshisundaram_properties_1949} used the Laplace transform and knowledge of the heat kernel to extend Carleman's idea. However, these methods did not provide precise error bounds.
	
	The optimal error was found by Levitan \cite{levitan_1952_asymptotic,levitan_1955_asymptotic_II} and Avakumovi$\check{\textrm{c}}$ \cite{avakumovic_uber_1956} by using the method of the cosine transform we adopted above. \textcite{hormander_spectral_1968} then introduced \textit{Fourier Integral Operators} as a tool to improve their approach and prove estimates for the spectral function of any positive elliptic pseudo-differential operator. In the case of closed manifolds, this method provided an optimal remainder estimate. The case of manifolds with boundary can also be treated with these tools, however obtaining the optimal bound of $\lambda^{(n-1)/2}$ is cumbersome to say the least, cf. the discussion preceding and following Theorem 17.5.9 in \cite{hormander1994analysispseudodifferential}. On the other hand, explicit bounds for the spectral function were obtained by \textcite{safarov_fourier_tauberian_2001} via a detailed analysis of the contributions of the interior and of the boundary. 
	
	The rôle of the geometry appeared more clearly in the work of Chazarain \cite{chazarain_formule_1974} and of Duistermaat and Guillemin \cite{duistermaat_1975_spectrum}. They proved a rigorous version of the Gutzwiller \textit{trace formula}, namely that the distributional kernel of the wave group, restricted to the diagonal, has an asymptotic expansion in powers of $\lambda$. This formula characterises the singularities of the spectral counting function, which appear at the times $T$ for which the manifold admits a closed geodesic of period $T$. By smearing with a test function localised near 0 and shrinking its support, one obtains the eigenvalue asymptotics. This question has also received more attention recently with the improvements and geometric techniques in \textcite{canzani_weyl_2023}.
	
	Remarkably, the above results can be generalised to wave-type operators on globally hyperbolic stationary spacetimes. In particular, Strohmaier and Zelditch \cite{strohmaier_2021_gutzwiller,strohmaier_2021_spectral} proved a relativistic version of the trace formula for the Klein-Gordon operator. It is interesting here to observe that the heat and wave group expansions have the same coefficients, although the first is not relativistically invariant and does not make sense in this more general setting.
	
	Finally, the question of eigenvalue asymptotics has been posed and answered for the (stationary) Schrödinger equation under a variety of conditions on the potential. We mention to this regard the foundational work of \textcite{li_schrodinger_1983} and \textcite{melrose_1982_scattering_trave_wave}, and the comprehensive book by \textcite{ivrii_1998_microlocal_spectral}, which includes many results due to its author. 
	
	\printbibliography
	
\end{document}